\documentclass[11pt,hyp]{nyjm}
\usepackage{hyperref}
\hypersetup{nesting=true,debug=true,naturalnames=true}
\usepackage{graphicx,amssymb,upref}

\let\<\langle
\let\>\rangle
\usepackage[all,pdf]{xy}
\UseComputerModernTips

\let\uml\"

\title{Recursive formulas for the motivic Milnor basis}

\author{Jonas Irgens Kylling}
\address{Department of Mathematics, University of Oslo, Norway}
\email{jonasik@math.uio.no}
\thanks{The author was partially supported by the RCN Frontier Research Group Project no.~250399.}



\newcommand{\A}{\mathcal{A}}
\newcommand{\ZZ}{\mathbb{Z}}
\newcommand{\PP}{\mathcal{P}}
\newcommand{\Hom}{\operatorname{Hom}}
\newcommand{\Ext}{\operatorname{Ext}}
\newcommand{\tensor}{\otimes}
\newcommand{\Spec}{\operatorname{Spec}}
\newcommand{\DeltaT}{\widetilde{\Delta}}
\newcommand{\Sq}{\operatorname{Sq}}

\usepackage{cite}

\usepackage[poorman]{cleveref}

\theoremstyle{theoremstyle}
\newtheorem{theorem}{Theorem}
\newtheorem*{theorem*}{Theorem}
\newtheorem{lemma}[theorem]{Lemma}

\newtheorem*{proposition*}{Proposition}
\newtheorem{corollary}[theorem]{Corollary}
\newtheorem*{corollary*}{Corollary}
\newtheorem{remark}[theorem]{Remark}
\newtheorem{remark*}{Remark}

\newtheorem{defn*}{Definition}
\theoremstyle{definition}

\theoremstyle{theoremstyle}

\begin{document}
 
\begin{abstract}  
  We give recursive formulas for the generating elements in the Milnor basis of the mod 2 motivic Steenrod algebra.
\end{abstract}
\maketitle
\tableofcontents

The Steenrod algebra plays an important role in topology in its capacity of acting on the singular cohomology ring of any topological space.
The motivic Steenrod algebra, introduced by Voevodsky in \cite{Voevodsky:power} (and developed further in \cite{Riou} and \cite{HKO}), shares many properties of its topological namesake. Its main applications include the proof of the Milnor conjecture \cite{Voevodsky:Z/2}, the Hopkins-Morel-Hoyois isomorphism \cite{Hoyois}, the Milnor conjecture on quadratic forms \cite{slices}, and computations of motivic stable stems \cite{Morel:CRAS}, \cite{DI:MASS} and \cite{April1}.

In this paper we will work over a field \cite{Voevodsky:power}, an essentially smooth scheme over a field \cite{HKO}, or more generally a base scheme with no points of residue characteristic two \cite{Spitzweck}.
Then the dual of the mod 2 motivic Steenrod algebra is a Hopf algebroid and is free on a monomial basis \cite{Voevodsky:power}.
Hence it is possible to consider the dual basis of its monomial basis. This is called the Milnor basis of the motivic Steenrod algebra.
In the classical setting \cite{Milnor}, the Milnor primitives in this basis are given by the recursive formula for $n \geq 0$:
\begin{equation}
  \label{top}
  Q_{n+1} = [Q_n, \Sq^{2^{n+1}}].
\end{equation}
By convention $Q_0$ is the Bockstein operator $\Sq^1$.
The bracket is standard notation for the commutator $[A, B] := AB - BA$.
The main result of this paper shows an analogous formula for the Milnor basis in the motivic Steenrod algebra:
\begin{equation}
  \label{mot}
  Q_{n+1} = [Q_n, \Sq^{2^{n+1}}] + \rho Q_n \Sq^{2^n} Q_{n-1}.
\end{equation}
The formula is a corollary of \Cref{main}.
Here  $n \geq 0,\, Q_{-1} := 0$,
$Q_0 = \Sq^1$ is the motivic Bockstein operator,
and $\rho$ is the class of $-1$ in
\[
H^{1,1}(\Spec k; \ZZ/2) \cong k^\times/(k^\times)^2.
\]
If $\rho = 0$ (i.e., if and only if $-1$ has a square root in $k$), then (\ref{mot}) is identical to (\ref{top}) (up to inverting $\tau$ and restricting to elements in weight 0).

A more indirect recursive formula for $Q_n$, involving the dual $\PP^n$ of $\xi_n$ (an element of the dual Steenrod algebra, see \Cref{dual}),
was shown by Voevodsky in \cite[Proposition 13.6]{Voevodsky:power}; that is,
\[
Q_n = [Q_0, \PP^n],\ n \geq 1.
\]
We provide a formula for calculating $\PP^n$ in \Cref{main2}.

The Milnor primitives $Q_0 = \Sq^1$ and $Q_1 = \Sq^2 \Sq^1 + \Sq^3$ are the same as their topological counterparts.
The primitive $Q_2 = [Q_1, \Sq^4] + \rho Q_0Q_1 \Sq^2$ is given in \cite[Example 13.7]{Voevodsky:power}.
In the mod $p > 2$ motivic Steenrod algebra the usual topological formula for $Q_n$ (\cite[Corollary 2]{Milnor}) remains true, since the subalgebra generated by the Bockstein operator and the power operations is isomorphic to the topological Steenrod algebra \cite[Theorem 10.3]{Voevodsky:power}.
Our recursive formula for $p = 2$ is hinted to prior to Lemma 5.13 in \cite{Voevodsky:Z/l}.

This paper is also an initial attempt to find a general product formula in the motivic Milnor basis. Such a product formula could for instance be used to extend Bob Bruner's $\Ext$-program \cite{Bruner} to the motivic setting. The same approach used below can be used to determine other products in the motivic Milnor basis.
\subsection*{Acknowledgements}
I am grateful to Bob Bruner and Paul Arne {\O}stv{\ae}r for advice and valuable suggestions.
I thank the anonymous referee for helpful comments and suggestions to improve the article.
\section{Recollections on the motivic Steenrod algebra}
We review results on the mod 2 motivic Steenrod algebra. These results are found in \cite{Voevodsky:power}, \cite{Voevodsky:Z/l}, \cite{Riou}, \cite{HKO} and \cite{Spitzweck}.
Henceforth we drop the ``mod 2'' when we speak of the mod 2 motivic Steenrod algebra, and mod 2 motivic cohomology.

The motivic Steenrod algebra $\A^{\star}$ is the algebra of graded bistable endomorphisms of motivic cohomology.
It is an algebra over the motivic cohomology of a point $H^{\star} := H^{\star}(\Spec k; \ZZ/2)$.
As an algebra over $H^\star$ it is generated by the motivic Steenrod squares, $\Sq^{i}$,
which are subject to the motivic Adem relations.
In fact, it is a free $H^{\star}$-module with basis the admissible polynomials $\Sq^I := \Sq^{i_0} \Sq^{i_{1}} \dots \Sq^{i_n} \dots$, where $I = (i_0, i_1, \dots)$, such that $i_j \geq 2 i_{j+1} \geq 0$.
The dual motivic Steenrod algebra is defined as $\A_\star := \Hom(\A^\star, H^\star)$.
The dual motivic Steenrod algebra is given by
\begin{equation}
\A_\star \cong \frac{H_\star[\tau_0, \tau_1, \ldots, \xi_1, \xi_2, \ldots]}
{\tau_i^2 + \tau \xi_{i+1} + \rho(\tau_{i+1} + \tau_0\xi_{i+1})},
\label{dual}
\end{equation}
where $\tau_i$ has degree $(2^{i+1} - 1, 2^i - 1)$ and $\xi_i$ has degree $(2^{i+1} - 2, 2^i - 1)$.
Here we use the standard convention $H_\star := H^{-\star}$.
Recall that $\tau$ is the generator of $H^{0,1} \cong \mu_2(k)$,
and $\rho$ is the class of $-1$ in $H^{1,1} \cong k^\times/(k^\times)^2$.
The dual is a Hopf algebroid with structure maps
\begin{align}
  \eta_L x &= x = \eta_R x, \quad x \in H_{1,1}, \nonumber \\
  \eta_L\tau &= \tau,\quad  \eta_R\tau = \tau + \rho\tau_0, \nonumber \\
  \epsilon\tau_i &= 0,\quad \epsilon\xi_i = 0, \nonumber \\
  \Delta x &= x \tensor 1,\quad x \in H_\star, \nonumber \\
  \Delta \xi_k &= \sum_{i=0}^k \xi_{k-i}^{2^i} \tensor \xi_i, \label{coprod1} \\
  \Delta \tau_{k} &= \tau_{k} \tensor 1 + \sum_{i=0}^{k} \xi_{k- i}^{2^i}\tensor \tau_i. \label{coprod2}
\end{align}
In the coproducts (\ref{coprod1}) and (\ref{coprod2}), $\xi_0 = 1$.
We recall the definition of the coproduct:
Evaluation induces a bilinear pairing
\begin{align*}
\langle -, -\rangle : \A_\star \tensor_{l,l} \A^\star  \to H^\star \\
\langle \alpha, \theta \rangle := \alpha(\theta).
\end{align*}
The tensor product $\tensor_{l,l}$ is formed with respect to the left module structure on both $\A^\star$ and $\A_\star$ (i.e., the module structure induced by $\eta_L$).
There is a bilinear pairing
\begin{align*}
  [-, -]:
  (\A_\star \tensor_{r,l} \A_\star)
  \tensor_{l, l}
  (\A^\star \tensor_{r,l} \A^\star)
  \to
  H^\star \\
  [\alpha_1 \tensor \alpha_2, \theta_1 \tensor \theta_2]
  := \langle\alpha_1, \theta_1 \langle\theta_2, \alpha_2 \rangle \rangle.
\end{align*}
The tensor products $\tensor_{r, l}$ are formed with respect to the right and left $H^\star$-module structures.
Then the coproduct is defined by the identity
\[
[ \Delta \alpha, \theta_1 \tensor \theta_2 ] = 
\langle \alpha, \theta_1 \theta_2 \rangle.
\]
Hence the coproduct is a map $\Delta : \A_\star \to \A_\star \tensor_{r,l} \A_\star$.

The dual motivic Steenrod algebra is a free $H_\star$-module with basis the monomials
\[
\tau(E)\xi(R) := \prod \tau_i^{\epsilon_i}\prod \xi_{i}^{r_i}.
\]
Here $E = (\epsilon_0, \dots, \epsilon_n)$, $\epsilon_i \in \{0, 1\}$,
and $R = (r_1, \dots, r_m)$ is a finite sequence of nonnegative integers.
We define $\widehat{\tau(E)\xi(R)}$ to be the dual of $\tau(E)\xi(r)$ in $\A^\star$ with respect to the monomial basis.
Then the set $\{\widehat{\tau(E)\xi(R)}\}_{E,R}$ forms a basis for $\A^\star$, called the Milnor basis.
We define $Q(E) := \widehat{\tau(E)}$, $\PP^R := \widehat{\xi(R)}$, $Q_i := \widehat{\tau_i}$ and $\PP^i := \widehat{\xi_i}$.
\begin{lemma}[\protect{\cite[Lemma 13.1, Propositions 13.2, 13.4]{Voevodsky:power}}]
  \label{known}
  We have:
  \begin{align*}
    \PP^{(n)} &= \Sq^{2n}, \\
    \widehat{\tau(E)\xi(R)} &= Q(E)\PP^R, \\
    Q(E) &= \prod Q_{i}^{\epsilon_i}, \\
    [Q_i, Q_j] &= 0, \\
    Q_i^2 &= 0.
  \end{align*}
\end{lemma}
\section{Proof of main result}
\begin{theorem}
  \label{main}
  The following equality holds in the motivic Steenrod algebra for $n \geq 0$:
  \[
  \Sq^{2^{n+1}}Q_{n} = Q_{n+1} + \sum_{k=0}^{n} \rho^k Q_{n-k}Q_{n-k+1}\dots Q_n \Sq^{2^{n+1 - k}}.
  \]
\end{theorem}
\begin{proof}
  The right hand side is written in the Milnor basis.
  Hence, it suffices to evaluate $\Sq^{2^{n+1}}Q_n$ on the monomials $\tau(E)\xi(R)$.
  Let $a_k$ be the sequence with $1$ in the $k$'th position and $0$'s elsewhere.
  For $0 \leq k \leq n$, let $b_{n,k}$ be the sequence with $1$'s in position $n-k$ to $n$ inclusive, and $0$'s elsewhere.
  With these definitions the equality in \Cref{main} takes the form
  \[
  \langle \tau(E)\xi(R), \Sq^{2^{n+1}} Q_n \rangle = \begin{cases}
    1 & \text{if}\ E=a_{n+1}, R=\mathbf{0}, \\
    \rho^i & \text{if}\ E=b_{n,i}, R=(2^{n-i}),\, 0 \leq i \leq n, \\
    0 & \text{otherwise}.
    \end{cases}
  \]
  Equivalently, the monomials $\tau(E)\xi(R)$, for pairs $(E, R)$ as above,
  are the only monomials for which an $H_\star$-multiple of $\xi_1^{2^n}\tensor \tau_n$ occurs in the coproduct.

  From (\ref{coprod1}) and (\ref{coprod2}) it follows that only the coproducts of elements $\tau_k$ contain terms of the form $\xi_1^j \tensor \tau_{i}$. Hence it suffices to consider the coproducts of monomials $\tau(E)$, as the factor $\xi(R)$ is completely determined by $\tau(E)$.
  
  Consider the ideals
  \begin{align*}
    I_n &:= (\xi_1^{2^{n}+1}, \xi_2, \xi_3, \dots, \tau_0\xi^{2^n}_1, \tau_0\xi_2, \tau_0\xi_3, \dots \tau_1, \tau_2, \dots), \\
    J_n &:= (\xi_1, \xi_2, \dots, \tau_0\tau_n, \tau_{n+1}, \tau_{n+2}, \dots),
  \end{align*}
  and the composition
  \[
  \DeltaT : \A_\star \xrightarrow{\Delta} \A_\star \tensor_{r,l} \A_\star \xrightarrow{\widetilde{(-)}} \A_\star/I_n \tensor_{r,l} \A_\star/J_n.
  \]
  Since $[I_n \tensor_{r,l} \A_\star, \Sq^{2^{n+1}} \tensor Q_n ] = 0$
  and $[\A_\star \tensor_{r,l} J_n, \Sq^{2^{n+1}} \tensor Q_n] = 0$
  we get an induced map $[\widetilde{(-)}, \Sq^{2^{n+1}} \tensor Q_n] : \A_\star/I_n \tensor_{r,l} \A_\star/J_n \to H_\star$,
  such that
  \[
  [M, \Sq^{2^{n+1}} \tensor Q_n] = [\widetilde{M}, \Sq^{2^{n+1}} \tensor Q_n].
  \]
  In particular,
  $[\Delta(M), \Sq^{2^{n+1}} \tensor Q_n] = [\DeltaT(M), \Sq^{2^{n+1}} \tensor Q_n]$.
  Hence it suffices to evaluate $\DeltaT$ on the monomial basis.

  With these definitions, the coproduct formula (\ref{coprod2}) implies
  \[
  \DeltaT (\tau_k) = 1 \tensor \tau_k + \xi_1^{2^{k-1}} \tensor \tau_{k-1},
  \]
  for $k \geq 1$.
  We have $\DeltaT(\tau_0) = 1 \tensor \tau_0 + \tau_0 \tensor 1$, but we will see that the last term can be ignored.

  We claim that $\DeltaT(\tau(E))$ contains a nonzero $H_\star$-multiple of $\xi_1^j \tensor \tau_n$
  if and only if $E = a_{n+1}$, or $E = b_{n,k}$ for some $0 \leq k \leq n$.
  Since $\DeltaT(\tau_{n+1}) = \xi_{1}^{2^{n}}\tensor \tau_n$,
  we get $\DeltaT(M) = 0$ for any monomial $M$ such that $\tau_{n+1}$ divides $M$ and $\tau_{n+1} \neq M$.
  Thus we may restrict to monomials $\tau(E)$, for $E=(\epsilon_0, \epsilon_1, \dots)$ where $\epsilon_i = 0$ for $i > n$.

  Consider $n \geq i_1 > i_2 > \dots > i_l$ and the sequence $E = (\epsilon_0, \epsilon_1, \dots)$ which is nonzero in the positions $i_j$
  (that is, $\epsilon_j = 1$ if and only if $j \in \{i_1, \dots, i_l\}$).
  By expanding the terms in the coproduct $\DeltaT(\tau(E))$ we get
  \begin{align*}
    \DeltaT(\tau_{i_1}\tau_{i_2} \dots \tau_{i_l}) =
    &(1 \tensor \tau_{i_1} + \xi_1^{2^{i_1-1}}\tensor\tau_{i_1-1}) \\
    &(1 \tensor \tau_{i_2} + \xi_1^{2^{i_2-1}}\tensor\tau_{i_2-1}) \\
    &\qquad\qquad\vdots \\
    &(1 \tensor \tau_{i_l} + \xi_1^{2^{i_l-1}}\tensor\tau_{i_l-1}).
  \end{align*}
  We claim that $\DeltaT(\tau(E))$ contains an $H_\star$-multiple of $\xi^{2^{j}}_1 \tensor \tau_{n}$ if and only if $i_1 = n, i_2 = n-1, \dots, i_{l-1} = n-(l-2), i_{l} = n - (l-1)$,
  in which case it contains the term $\rho^{l-1}\xi_1^{2^{n} - 2^{n-l+1}} \tensor \tau_{n}$.
  Indeed, any term is of the form
  \[
  \prod_{j \in J} \xi_{1}^{2^{{i_j}-1}} \tensor \prod_{j \in J} \tau_{i_j - 1} \prod_{j \in [1, l]\setminus J} \tau_{i_j}
  \]
  for some $J \subseteq [1, l]$.
  If $i_l = 0$ there are also terms of the form
  \[\tau_0 \prod_{j \in J'} \xi_1^{2^j} \tensor \prod_{j \in J''} \tau_j.\]
  These terms can be ignored as no multiple of $\tau$ appears in $\prod_{j \in J''} \tau_j$.
  The theorem follows now from \Cref{induction} below.
\end{proof}
\begin{lemma}
  \label{induction}
  Consider a sequence of $l \geq 2$ nonnegative integers $\{j_i\}_{i=1}^{l}$ such that $n > j_1 \geq j_2 \geq \dots \geq j_l$ and $j_i > j_{i+2}$.

  Then the product $\tau_{j_1}\dots \tau_{j_l}$ is an $H_\star$-multiple of $\tau_{j_1+1}$ modulo $J_n$ if and only if
  $j_2 = j_1 - 1, \dots, j_{i} = j_1 - i, \dots, j_{l-1} = j_1 - (l-1), j_{l} = j_1-(l-1)$,
  in which case the product equals $\rho^{l-1} \tau_{j_1 + 1}$, and we say that the sequence is decreasing.
  Otherwise the product equals $\rho^p\tau_{j_1'}\dots\tau_{j_k'}$,
  for some $p \geq 0 $, $k \geq 2$ and $j_1 + 1 \geq j_1' > \dots > j_k'$.
\end{lemma}
\begin{proof}
  Induction on the length $l$.
  If $l = 2$ the statement is immediate from the relation $\tau_{j_1}^2 = \rho \tau_{j_1+1}$ mod $J_n$.
  Assume $l \geq 3$ and the statement is true for sequences of length $l-1$.
  There are three case distinctions:
  \begin{itemize}
  \item If $j_1 > j_2$ and $\{j_i\}_{i=2}^{l}$ is not decreasing:
    Then
    \[\tau_{j_2} \dots \tau_{j_l} = \rho^p\tau_{j_2'}\dots\tau_{j_k'}\] for some $j_i'$ and $k \geq 3$.
    After multiplication by $\tau_{j_1}$ there are still at least $(k-1)$ $\tau_i$-factors.
  \item If $j_1 > j_2$ and $\{j_i\}_{i=2}^{l}$ is decreasing:
    Then \[\tau_{j_1}\tau_{j_2} \dots \tau_{j_l} = \rho^{l-2}\tau_{j_1}\tau_{j_2+1} = \rho^{k-1}\tau_{j_1 + 1}\]
    if and only if $j_1 = j_2 + 1$.
    That is, if and only if $\{j_i\}_{i=1}^{l}$ is decreasing.
  \item If $j_1 = j_2$:
    Then $\tau_{j_1} \dots \tau_{j_l} = \tau_{j_1}^2 \prod_{i=3}^l \tau_{j_i} = \rho \tau_{j_1+1}\prod_{i=3}^l \tau_{j_i}$.
    The last factor $\prod_{i=3}^l \tau_{j_i}$ equals $\tau_{j_3}$,
    $\rho^{l-3}\tau_{j_3+1}$ or a product with multiple $\tau_i$-factors.
    In either case the product has multiple $\tau_i$-factors, since $j_1 > j_3$.
  \end{itemize}
\end{proof}
\begin{corollary}
  The Milnor primitives are generated by the recursive formula
  \[
  Q_{n+1} = [Q_n, \Sq^{2^{n+1}}] + \rho Q_n \Sq^{2^n} Q_{n-1}.
  \]
\end{corollary}
\begin{proof}
  Apply \Cref{main} to $\Sq^{2^{n+1}}Q_n$ and $\Sq^{2^n}Q_{n-1}$.
\end{proof}
\begin{remark}
  Using the same techniques as in \Cref{main} we can prove the following formulas for $n \geq 0$:
  \begin{align*}
    [Q_n, Sq^{2^i}] =& 0,\quad  i < n, \\
    [Q_n, Sq^{2^n}] =& \rho Q_n Q_{n-1}, \\
    [Q_n, Sq^{2^{i}}] =& Q_{n+1}\Sq^{2^i-2^{n+1}} \\
    &+ \sum_{k=1}^n \rho^{k} Q_{n-k}Q_{n-k+1}\dots Q_{n} \Sq^{2^{i} + 2^{n+1 - k} - 2^{n+1}}, \quad i > n.
  \end{align*}
  In particular, we can write the recursive formula
  \[
  Q_{n + 1} = [Q_n, \Sq^{2^{n+1}}] + \rho \Sq^{2^{n}}Q_n Q_{n-1}.
  \]
\end{remark}
The first few primitives are:
\begin{align*}
  Q_2 &= [Q_1, \Sq^4] + \rho Q_1 Q_0 \Sq^{2} \\
  Q_3 &= [Q_2, \Sq^8] + \rho Q_2 Q_1 \Sq^{4} + \rho^2 Q_2 Q_1 Q_0 \Sq^{2} \\
  Q_4 &= [Q_3, \Sq^{16}] + \rho Q_3 Q_2 \Sq^{8} + \rho^2 Q_3 Q_2 Q_1\Sq^{4} + \rho^3 Q_3 Q_2 Q_1 Q_0\Sq^2{}.
\end{align*}
In the Cartan basis the primitives $Q_0$ through $Q_6$ have 1, 2, 5, 14, 47, 213 and 1427 nonzero terms, respectively.
\section{A formula for $\PP^n$}
For $n \geq 1$ we can compute $\PP^n$ recursively using the formula
\begin{align}
  \label{eq:qn}
  [\PP^n, \Sq^{2^{n+1}}]
  & = \PP^{n+1} + \tau Q_n Q_0 \Sq^{2^{n+1} - 2} \\
  & + \sum_{i=0}^{n-1}\tau\rho^i Q_{n-i-1}Q_{n-i}\dots Q_n (\Sq^{2^{n-i}} + \rho Q_0 \Sq^{2^{n-i} - 2}). \nonumber
\end{align}
\Cref{eq:qn} is a combination of \Cref{main2} and \Cref{lem:qSq} below.
Note that $q_1 = \Sq^2$.

\begin{theorem}
  \label{main2}
  For $n \geq 1$ we have the formula
\begin{align}
  \label{eq:qn1}
  Sq^{2^{n+1}}\PP^n &= \PP^{n+1} + \widehat{\xi_1^{2^n}\xi_n} \\
  & + \tau \sum_{i=0}^{n-1}\rho^{i}Q_{n-i-1}Q_{n-i}\dots Q_n \Sq^{2^{n-i}} \nonumber\\
  & + \tau \rho \sum_{i=0}^{n-1}\rho^{i}Q_{n-i-1}Q_{n-i}\dots Q_n Q_0 \Sq^{2^{n-i} - 2}.\nonumber
\end{align}

\end{theorem}
\begin{remark}
  Note that the term for $i = n-1$ in the last sum is zero,
  and the term for $i=n-2$ cancel with the term for $i=n-1$ in the first sum.
  That is, the terms on the right hand side are not linearly independent.
\end{remark}
\begin{proof}
  \Cref{eq:qn1} is equivalent to 
  \[
  \langle \tau(E)\xi(R), \Sq^{2^{n+1}} \PP^n \rangle = \begin{cases}
    1 & \text{if}\ E=0, R=a_{n+1} \text{ or } a_n + (2^n), \\
    \tau \rho^i & \text{if}\ E=b_{n,i}, R=(2^{n-i - 1}),\, 0 \leq i < n, \\
    \tau \rho^{i+1} & \text{if}\ E=b_{n,i} + (1), R=(2^{n-i - 1} - 1),\, 0 \leq i < n-1, \\
    0 & \text{otherwise}.
    \end{cases}
  \]
  
  Consider the ideals,
  \begin{align*}
    I_n &:= (\xi_1^{2^{n}+1}, \xi_2, \xi_3, \dots, \tau_1, \tau_2, \dots), \\
    J_n &:= (\xi_1, \xi_2, \dots, \xi_{n-1}, \xi_n^2, \xi_{n+1}, \dots, \tau_0\xi_n, \dots \tau_{n-1}\xi_n, \tau_{n}, \tau_{n+1}, \dots).
  \end{align*}
  Then $[I_n \tensor_{r,l} \A_\star, \Sq^{2^{n+1}} \tensor \PP^n ] = 0$
  and $[\A_\star \tensor_{r,l} J_n, \Sq^{2^{n+1}} \tensor \PP^n] = 0$,
  and we reduce to consider the composition
  \[
  \DeltaT : \A_\star \xrightarrow{\Delta} \A_\star \tensor_{r,l} \A_\star \xrightarrow{\widetilde{(-)}} \A_\star/I_n \tensor_{r,l} \A_\star/J_n.
  \]
  With these definitions the coproducts are
  \begin{align*}
    \DeltaT (\xi_n) &= 1 \tensor \xi_n, \\
    \DeltaT (\xi_{n+1}) &= \xi_1^{2^{n}} \tensor \xi_{n}, \\
    \DeltaT (\xi_k) &= 0,\quad k\neq 1, n, n+1, \\
    \DeltaT (\tau_k) &= 1 \tensor \tau_k + \xi_1^{2^{k-1}} \tensor \tau_{k-1},
  \end{align*}
  for $k \geq 1$.
  The pairings with $\xi_{n+1}$ and $\xi_{1}^{2^n}\xi_{n}$ are immediate.

  Since $\DeltaT(\xi_k) = 0$ for $k \neq 1, n,n+1$ we may restrict to consider the coproducts of monomials $\tau(E)$,
  as the factor $\xi(R) = \xi_{1}^{p}$ is completely determined by $\tau(E)$.

  Let $M$ be the monomial $\tau(E)$.
  Since $\DeltaT(\tau_{k}) = 0$, when $ k > n$,
  we may suppose $E=(\epsilon_0, \epsilon_1, \dots)$ where $\epsilon_i = 0$ for $i > n$.

  Consider $n \geq i_1 > i_2 > \dots > i_l$ and the sequence $E = (\epsilon_0, \epsilon_1, \dots)$ which is nonzero in the positions $i_j$.
  Any term in the coproduct $\DeltaT(\tau(E))$ is of the form
  \begin{equation}
    \label{term1}
    \prod_{j \in J} \xi_{1}^{2^{{i_j}-1}} \tensor \prod_{j \in J} \tau_{i_j - 1} \prod_{j \in J'} \tau_{i_j},
  \end{equation}
  for some $J \subseteq [1, l]$ and $J' := [1, l]\setminus J$.
  If $i_l=0$ we also have terms of the form
  \begin{equation}
    \label{term2}
    \tau_0 \prod_{j \in J} \xi_{1}^{2^{{i_j}-1}} \tensor \prod_{j \in J} \tau_{i_j - 1} \prod_{j \in J'} \tau_{i_j},
  \end{equation}
  for some $J \subseteq [1, l-1]$, and $J' := [1, l-1]\setminus J$.

  By \Cref{induction2} the last factor of the tensor product,
  $\prod_{j \in J} \tau_{i_j - 1} \prod_{j \in J'} \tau_{i_j}$,
  contains a multiple $\xi_n$ if and only if 
  the sequence $i_1 - \chi_J(1) \geq \dots \geq i_{l'} - \chi_J(l')$ is decreasing (in the sense of \Cref{induction2}) and $n-1 = i_1 - \chi_J(1)$
  (here $l' = l$ or $l-1$, and $\chi_J$ is the characteristic function of $J$, i.e. $\chi(i) = 1$ if $i \in J$ and 0 otherwise).
  That is, if and only if $J = [1, l')$ and $i_1 = n$. Assume that $E$ is such a sequence.
  Then
  \begin{align*}
  \prod_{j \in J} \xi_{1}^{2^{{i_j}-1}} \tensor \prod_{j \in J} \tau_{i_j - 1} \prod_{j \in J'} \tau_{i_j}
  &=
  \xi_1^{2^{n-1} - 2^{n - l'}} \tensor \tau \rho^{l' - 2}\xi_n \\
  &=
  \rho^{l'-2}(\tau + \rho\tau_0)\xi_1^{2^{n-1} - 2^{n - l'}} \tensor \xi_n.
  \end{align*}
  There are three case distinctions:
  \begin{itemize}
    \item If $i_l \neq 0$: Then there are only terms of the form of \Cref{term1}, and
      the term $\tau\rho^{l-2}\xi_1^{2^{n-1} - 2^{n - (l-1)}} \tensor \xi_n$ is the only term of the coproduct
      which can contribute with a nonzero pairing.
      This gives the pairing for monomials with $E = b_{n,l}$ for $0 \leq l < n$.
    \item If $i_{l-1} - 1 > i_l = 0$: Then there are terms of the form of both \Cref{term1} and \Cref{term2},
      however only the latter is of the form $a\xi_1^x \tensor \xi_n$. It equals,
      $\tau\rho^{l-2}\xi_1^{2^{n-1} - 2^{n - (l-2)}+1} \tensor \xi_n$.
      This gives the pairing for monomials with  $E = (1) + b_{n, l}$ for $0 \leq l < n-1$.
    \item If $i_{l-1} - 1 = i_l = 0$, i.e., $E = b_{n, n}$:
      Then terms of the form of both \Cref{term1} and \Cref{term2} contribute and cancel each other.
      More precisely, the term with $J=[1, l]$ from \Cref{term1} cancels the term with $J=[1, l-1]$ from \Cref{term2}.
      Hence the pairing is zero in this case.
  \end{itemize}
\end{proof}
\begin{lemma}
  \label{induction2}
  Consider a sequence of $l \geq 2$ nonnegative integers $\{j_i\}_{i=1}^{l}$ such that $n > j_1 \geq j_2 \geq \dots \geq j_l$ and $j_i > j_{i+2}$.
  
  Then the product $\tau_{j_1}\dots \tau_{j_l}$ is an $H_\star$-multiple of $\tau_{j_1+1}$ or $\xi_n$ modulo $J_n$ if and only if
  $j_2 = j_1 - 1, \dots, j_{i} = j_1 - i, \dots, j_{l-1} = j_1 - (l-1), j_{l} = j_1-(l-1)$,
  in which case the product equals $\rho^{l-1} \tau_{j_1 + 1}$ or $\rho^{l-2}\xi_{n}$, and we say that the sequence is decreasing.
  The product equals $\rho^{l-2}\xi_{n}$ only if $j_1 = n-1$.
  Otherwise the product equals $\rho^p\tau_{j_1'}\dots\tau_{j_k'}\tau^\epsilon\xi_{n}^{\epsilon}$,
  for some $p \geq 0 $, $\epsilon \in \{0, 1\}$, $k \geq 2$, and $j_1+1 \geq j_1' > \dots > j_k'$.
\end{lemma}
\begin{proof}
  Induction on the length $l$.
  If $l = 2$ the statement is immediate from the relation
  \[
  \tau_{j_1}^2 = \begin{cases}
    \tau \xi_{n} & j_1 = n-1 \\
    \rho \tau_{j_1+1} & \text{otherwise}
  \end{cases}
  \quad\bmod J_n.
  \]
  Assume $l \geq 3$ and the statement is true for sequences of length $l-1$.
  There are three case distinctions:
  \begin{itemize}
  \item If $j_1 > j_2$ and $\{j_i\}_{i=2}^{l}$ is not decreasing:
    Then \[\tau_{j_2} \dots \tau_{j_l} = \rho^p\tau_{j_2'}\dots\tau_{j_k'}\tau^\epsilon\xi_{n}^{\epsilon},\] for some $j_i'$ and $k \geq 3$.
    After multiplication by $\tau_{j_1}$ there are still at least $(k-1)$ $\tau_i$-factors.
  \item If $j_1 > j_2$ and $\{j_i\}_{i=2}^{l}$ is decreasing:
    Then $\tau_{j_1}\tau_{j_2} \dots \tau_{j_l} = \rho^{l-2}\tau_{j_1}\tau_{j_2+1}$ equals $\rho^{l-1}\tau_{j_1 + 1}$
    or $\tau\rho^{l-2}\xi_{n}$
    if and only if $j_1 = j_2 + 1$.
    That is, if and only if $\{j_i\}_{i=1}^{l}$ is decreasing.
  \item If $j_1 = j_2$:
    Then $\tau_{j_1} \dots \tau_{j_l} = \tau_{j_1}^2 \prod_{i=3}^l \tau_{j_i}$
    which is $\rho \tau_{j_1+1}\prod_{i=3}^l \tau_{j_i}$ or $\tau \xi_{n}\prod_{i=3}^l \tau_{j_i}$.
    The last factor $\prod_{i=3}^l \tau_{j_i}$ equals $\tau_{j_3}$, $\rho^{l-3}\tau_{j_3+1}$
    or a product with multiple $\tau_i$-factors.
    In either case the product has multiple $\tau_i$-factors, since $j_1 > j_3$.
  \end{itemize}
\end{proof}
\begin{lemma}
  \label{lem:qSq}
  \[
  \PP^n\Sq^{2^{n+1}} = \widehat{\xi_1^{2^n}\xi_n} + \tau Q_n Q_0 \Sq^{2^{n+1} - 2}.
  \]
\end{lemma}
\begin{proof}
  This is a simpler variant of \Cref{main} and \Cref{main2}
  using the ideals
  \begin{align*}
    I_n &:= (\xi_1, \dots, \xi_{n-1}, \xi_n^2, \xi_{n+1}, \dots, \tau_{n}, \tau_{n+1}, \dots), \\
    J_n &:= (\xi_{1}^{2^{n}+1}, \xi_{2}, \xi_{3}, \dots, \tau_{1}, \tau_{2}, \dots).
  \end{align*}
\end{proof}
\begin{remark}
  We have not been able to generalize Milnor's product formula for arbitrary elements in the Milnor basis.
  \Cref{main} and \Cref{main2} suggest that such a formula is more involved than the topological version.
  The complexity is due to the relation $\tau_i^2 = (\tau + \rho\tau_0)\xi_{i+1} + \rho \tau_{i+1}$,
  and the twisted tensor product induced by the right action $\eta_R(\tau) = \tau + \rho\tau_0$.
\end{remark}

\bibliography{main}{}
\bibliographystyle{plain}

\end{document}